\documentclass[11pt,letterpaper]{article}
\usepackage[utf8]{inputenc}
\usepackage[margin=1in]{geometry}
\usepackage{amsmath, amssymb, amsthm}
\usepackage[ruled,linesnumbered]{algorithm2e}
\usepackage[dvipsnames]{xcolor}
\usepackage{bbm}
\usepackage[normalem]{ulem}
\usepackage{graphicx}
\usepackage{rotating}
\usepackage{subcaption}
\usepackage{booktabs}
\usepackage{graphicx}
\usepackage{multirow}
\usepackage{float}
\usepackage{subfiles}
\usepackage{enumitem}
\usepackage{tikz-network}
\usepackage{caption}
\usepackage{subcaption}

\DeclareMathOperator*{\maximize}{maximize}

\DeclareMathOperator*{\argmax}{arg\,max}

\DeclareMathOperator*{\mF}{\mathcal{F}}

\DeclareMathOperator*{\ALG}{\text{ALG}}

\definecolor{cof}{RGB}{219,144,71}
\definecolor{pur}{RGB}{186,146,162}
\definecolor{greeo}{RGB}{91,173,69}
\definecolor{greet}{RGB}{52,111,72}

\usepackage{color}

\newtheorem{theorem}{Theorem}[section]

\newtheorem{proposition}[theorem]{Proposition}
\newtheorem{remark}[theorem]{Remark}
\newtheorem{example}[theorem]{Example}

\title{Permutatorial Optimization via the Permutahedron}
\author{
J. Carlos Mart\'{i}nez Mori\thanks{Corresponding author. Center for Applied Mathematics, Cornell University. Contact: \texttt{jm2638@cornell.edu}.} 
\and 
Samitha Samaranayake\thanks{School of Civil and Environmental Engineering, Cornell University. Contact: \texttt{samitha@cornell.edu}.}
}
\date{}

\begin{document}

\maketitle

\vspace{-1cm}
\begin{abstract}
A water company decides to expand its network with a set of water lines, but it cannot build them all at once. 
However, it starts reaping benefits from a partial expansion.
In what order should the company build the lines? 
We formalize a class of permutatorial problems with combinatorial/continuous subproblems capturing applications of incremental deployment. 
We show that, for additive/linear objective functions, efficient polyhedral methods for the subproblems extend to the permutatorial problem. 
Our main technical ingredient is the permutahedron.
\vspace{0.1cm}

\end{abstract}

\section{Introduction}
\label{sec: introduction}
Let $E = \{1, 2, \ldots, m\}$ be a ground set of elements and let $\mF \subseteq 2^E$ be a collection of subsets of the ground set.
For brevity we denote $[m] := \{1, 2, \ldots, m\}$.
Let $f: 2^E \rightarrow \mathbb{R}$ be a set function.
Our starting point is combinatorial problems of the form
\begin{align}
\label{eq: combinatorial optimization}
    \max\{f(S): S \in \mF\},
\end{align}
where $\mF$ is the feasible region and $f$ is the objective function.
Let $\Pi$ be the set of permutations of $E$\textemdash the bijections from $E$ onto itself.
Note that each permutation $\pi \in \Pi$ induces a distinct maximal chain $\emptyset = H^0(\pi) \subset H^1(\pi) \subset H^2(\pi) \subset \cdots \subset H^m(\pi) = E$ of subsets of $E$ ordered by inclusion, where $H^j(\pi) := \{i \in E: \pi(i) \leq j\}$.
For a subset $H \subseteq E$ of elements, let $\mF(H) := \{S \subseteq H : S \in \mF \}$ be the restriction of $\mF$ to $H$.
We consider permutatorial problems of the form
\begin{align}
\label{eq: permutatorial optimization - combinatorial}
    \max\{g(\pi) : \pi \in \Pi\},
\end{align}
where $g: \Pi \rightarrow \mathbb{R}$ is the permutation function given by
\begin{align*}
    g(\pi) := \sum_{j=1}^m \max\{f(S): S \in \mF(H^j(\pi)) \}.
\end{align*}
This form captures problems where a set of elements to be realized has been fixed, but the order in which they are realized is to be determined.
The objective function sums over the objective values of a sequence of \emph{decoupled} combinatorial subproblems of the form (\ref{eq: combinatorial optimization}), each of which may only use elements realized by their corresponding step.
At any given step, whether a combination of realized elements can be used is still subject to $\mF$.
Our main result is that, for $f$ additive, efficient polyhedral methods for the combinatorial subproblem (\ref{eq: combinatorial optimization}) extend to the permutatorial problem (\ref{eq: permutatorial optimization - combinatorial}). 
\begin{theorem}
\label{theorem: main - combinatorial}
If $f$ is additive and (\ref{eq: combinatorial optimization}) can be optimized in polynomial time via the polyhedral approach, then (\ref{eq: permutatorial optimization - combinatorial}) can be optimized in polynomial time via the polyhedral approach.
\end{theorem}

As a motivational example, consider a water company that has decided to expand its network with a set of water lines.
The company cannot build the water lines all at once, but it starts reaping benefits from a partial network expansion (e.g., increasing the amount of water transmitted from one point to another).
The goal is to find an ordering in which to build the water lines so that the cumulative benefit reaped over time, including during the construction phase, is maximized\textemdash we formalize this example shortly.
This type of permutatorial consideration is important in the strategic planning of many sorts of infrastructure projects, where phased construction spanning multiple years is often unavoidable.
It is moreover timely in the United States with the recently enacted infrastructure bill and its funding provisions for transportation, broadband, water, and power infrastructure projects~\cite{bill}.

Our technique more generally applies if we replace the combinatorial subproblem (\ref{eq: combinatorial optimization}) by a continuous subproblem of the form
\begin{align}
\label{eq: continuous optimization}
    \max\{F(x): x \in P \},
\end{align}
where $P \subseteq \mathbb{R}^{n}$ is a polyhedral feasible region and $F: \mathbb{R}^{n} \rightarrow \mathbb{R}$ is the objective function.
In this case, the corresponding permutatorial problem is
\begin{align}
\label{eq: permutatorial optimization - continuous}
    \max\{g(\pi) : \pi \in \Pi\},
\end{align}
where $g: \Pi \rightarrow \mathbb{R}$ is the permutation function given by
\begin{align*}
    g(\pi) := \sum_{j=1}^m \max\{F(x): x \in P(\chi_{H^j(\pi)}) \},
\end{align*}
where $\chi_{H^j(\pi)} \in \{0,1\}^m$ is the characteristic vector corresponding to $H^j(\pi)$ and $P(\chi_{H^j(\pi)})$ is a linear ``restriction'' of $P$ to $\chi_{H^j(\pi)}$.
The precise form of this ``restriction'' is problem-specific, but in any case it ensures that $x \in P(\chi_{H^j(\pi)})$ may only use elements realized by the $j$th step and is still subject to $P$.
In particular, if $P \subseteq [0,1]^n$ and $n = m$, then $P(\chi_{H^j(\pi)}) := \{x \in \mathbb{R}^m: x \leq \chi_{H^j(\pi)} \land x \in P \}$ is a natural choice.
See two concrete examples in Section~\ref{sec: algorithm for additive/linear functions}.
We show that, for $F$ a linear function, efficient polyhedral methods for the continuous subproblem (\ref{eq: continuous optimization}) extend to the permutatorial problem (\ref{eq: permutatorial optimization - continuous}).
We in turn use this to prove our main result.

\subsection{Greedy Can Fail}
\label{sec: greedy can fail}

To motivate our focus on polyhedral methods, consider the performance of two natural greedy algorithms for~(\ref{eq: permutatorial optimization - combinatorial}).
The most natural greedy algorithm maintains a set $H$ of realized elements and, on every step, brings $e^* := \argmax_{e \in E \setminus H} \{\max\{f(S): S \in \mF(H + e) \} - \max\{f(S): S \in \mF(H) \}\}$ into $H$.
Another natural greedy algorithm first finds $S^* := \argmax \{f(S): S \in \mF\}$ and greedily brings its elements into $H$.
It then brings any elements remaining in $E \setminus S^*$ into $H$.
As we showcase in the following examples, neither algorithm is optimal, even when $f$ is an additive set function.

\begin{example}[Matchings]
\label{example: matchings}
Consider the following weighted graphs with $\epsilon > 0$ small.
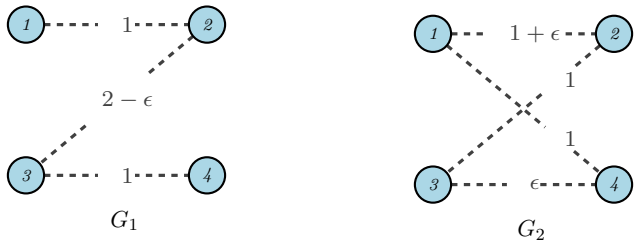
\begin{center}
\begin{minipage}{.19\linewidth}
\resizebox{\linewidth}{!}{
    \centering
    \begin{tikzpicture}
    \Text[x=1.5, y=-2]{$G_1$}
    \Vertex[x=0,y=1.25,label=1]{1}
    \Vertex[x=3,y=1.25,label=2]{2}
    \Vertex[x=0,y=-1.25,label=3]{3}
    \Vertex[x=3,y=-1.25,label=4]{4}
    \Edge[Math,label=1, fontscale=4/3,style={dashed}](1)(2)
    \Edge[Math, label=2-\epsilon, fontscale=4/3,style={dashed}](2)(3)
    \Edge[Math, label=1, fontscale=4/3,style={dashed}](3)(4)
    \end{tikzpicture}
}
\end{minipage}
\hspace{2cm}
\begin{minipage}{.19\linewidth}
\resizebox{\linewidth}{!}{
    \centering
    \begin{tikzpicture}
    \Text[x=1.5, y=-2]{$G_2$}
    \Vertex[x=0,y=1.25,label=1]{1}
    \Vertex[x=3,y=1.25,label=2]{2}
    \Vertex[x=0,y=-1.25,label=3]{3}
    \Vertex[x=3,y=-1.25,label=4]{4}
    \Edge[Math,label=1+\epsilon, fontscale=4/3,style={dashed}](1)(2)
    \Edge[Math, label=1, distance=0.25, fontscale=4/3,style={dashed}](2)(3)
    \Edge[Math, label=\epsilon, fontscale=4/3,style={dashed}](3)(4)
    \Edge[Math, label=1, distance=0.75, fontscale=4/3,style={dashed}](1)(4)
\end{tikzpicture}
}
\end{minipage}
\end{center}
Suppose the dashed edges are realized one at time. 
The problem is to find an ordering in which to realize the dashed edges so that the sum of the weights of the maximum weight matchings over the steps, each of which may only use the edges realized by their corresponding step, is maximized.
Note that this problem is captured by~(\ref{eq: permutatorial optimization - combinatorial}): $\mF$ represents the set of matchings in $G$, $f(S)$ represents the sum of the weights of the edges in $S$, and $\mF(H^j(\pi))$ represents the set of matchings in $G$ when only the edges in $H^j(\pi)$ can be used.

On $G_1$, the first algorithm is optimal, realizing the edges $\{2,3\}$, $\{1,2\}$, and $\{3,4\}$ in that order for a cumulative weight of $6 - 2 \epsilon$.
Conversely, the second algorithm is suboptimal, realizing the edges $\{1,2\}$, $\{3,4\}$, and $\{2,3\}$ in that order for a cumulative weight of $5$.
Therefore, the second greedy algorithm cannot achieve an approximation factor better than $5/6$ for this problem.
On $G_2$, the first algorithm is suboptimal, realizing the edges $\{1,2\}$, $\{3,4\}$, $\{1,4\}$ and $\{2,3\}$ in that order for a cumulative weight of $5 + 5\epsilon$.
Conversely, the second algorithm is optimal, realizing the edges $\{1,4\}$, $\{2,3\}$, $\{1,2\}$, and $\{3,4\}$ in that order for a cumulative weight of $7$.
Therefore, the first greedy algorithm cannot achieve an approximation factor better than $5/7$ for this problem.
$G_1$ and $G_2$ can be combined to produce an instance in which neither algorithm is optimal (e.g., $G_1 \cup G_2)$.
\end{example}
We note that if $\mF = 2^E$ (i.e., unconstrained optimization) and $f$ is monotone submodular, then $E = \argmax \{f(S): S \in \mF\}$ and so the two greedy algorithms are equivalent.
In this special case, the greedy algorithm can be shown to be a $1/e$-approximation algorithm to (\ref{eq: permutatorial optimization - combinatorial}) using standard techniques.
We show this in Section~\ref{sec: conclusion}.
Next, we formalize our motivational example and consider the performance of the analogous greedy algorithms for (\ref{eq: permutatorial optimization - continuous}).
\begin{example}[Flows]
\label{example: flows}
Let $D = (V,A)$ be a directed graph with capacities $c: A \rightarrow \mathbb{Q}_{\geq 0}$.
The arcs represent, for example, capacitated water lines planned for construction.
Let $s, t \in V$ be distinct nodes.
Suppose a water company wants to maximize the amount of water transmitted from $s$ to $t$ over time, including during the construction phase.
Consider the problem on the following capacitated directed graphs with $\epsilon > 0$ small, which we adapt from Example~\ref{example: matchings}. 
Unlabeled arcs are uncapacitated.
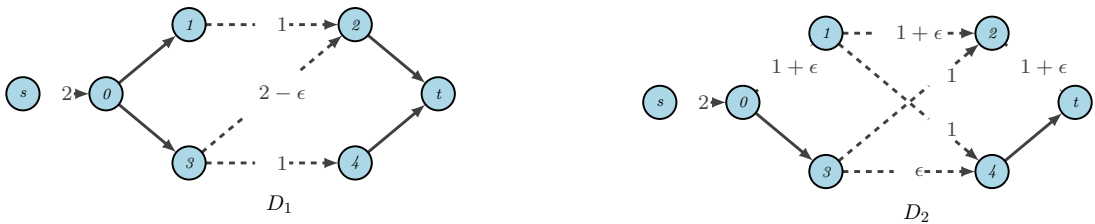
\begin{center}
\begin{minipage}{.375\linewidth}
\resizebox{\linewidth}{!}{
    \centering
    \begin{tikzpicture}
    \Text[x=1.5, y=-2]{$D_1$}
    \Vertex[x=-3, y=0, label=s]{s}
    \Vertex[x=-1.5, y=0, label=0]{0}
    \Vertex[x=0, y=1.25, label=1]{1}
    \Vertex[x=3, y=1.25, label=2]{2}
    \Vertex[x=0, y=-1.25, label=3]{3}
    \Vertex[x=3, y=-1.25, label=4]{4}
    \Vertex[x=4.5, y=0, label=t]{t}
    \Edge[Direct, Math, label=2, distance=1/3, fontscale=4/3](s)(0)
    \Edge[Direct, Math, fontscale=4/3](0)(1)
    \Edge[Direct, Math, fontscale=4/3](0)(3)
    \Edge[Direct, Math,label=1, fontscale=4/3, style={dashed}](1)(2)
    \Edge[Direct, Math, label=2-\epsilon, fontscale=4/3, style={dashed}](3)(2)
    \Edge[Direct, Math, label=1, fontscale=4/3, style={dashed}](3)(4)
    \Edge[Direct, Math, fontscale=4/3](2)(t)
    \Edge[Direct, Math, fontscale=4/3](4)(t)
    \end{tikzpicture}
}
\end{minipage}
\hspace{2cm}
\begin{minipage}{.375\linewidth}
\resizebox{\linewidth}{!}{
    \centering
    \begin{tikzpicture}
    \Text[x=1.5, y=-2]{$D_2$}
    \Vertex[x=-3, y=0, label=s]{s}
    \Vertex[x=-1.5, y=0, label=0]{0}
    \Vertex[x=0, y=1.25, label=1]{1}
    \Vertex[x=3, y=1.25, label=2]{2}
    \Vertex[x=0, y=-1.25, label=3]{3}
    \Vertex[x=3, y=-1.25, label=4]{4}
    \Vertex[x=4.5, y=0, label=t]{t}
    \Edge[Direct, Math, label=2, distance=1/3, fontscale=4/3](s)(0)
    \Edge[Direct, Math, label=1+\epsilon,  fontscale=4/3](0)(1)
    \Edge[Direct, Math, fontscale=4/3](0)(3)
    \Edge[Direct, Math, label=1+\epsilon, fontscale=4/3, style={dashed}](1)(2)
    \Edge[Direct, Math, label=1, distance=0.75, fontscale=4/3, style={dashed}](3)(2)
    \Edge[Direct, Math, label=\epsilon, fontscale=4/3, style={dashed}](3)(4)
    \Edge[Direct, Math, label=1, distance=0.75, fontscale=4/3, style={dashed}](1)(4)
    \Edge[Direct, Math, label=1+\epsilon, fontscale=4/3](2)(t)
    \Edge[Direct, Math, fontscale=4/3](4)(t)
    \end{tikzpicture}
}
\end{minipage}
\end{center}
Suppose the dashed arcs are realized one at time. 
The problem is to find an ordering in which to realize the dashed arcs so that the sum of the maximum flows from $s$ to $t$ over the steps, each of which may only use the arcs realized by their corresponding step, is maximized.
Note that this problem is captured by~(\ref{eq: permutatorial optimization - continuous}): $P$ represents the $s-t$ flow polytope of $D$, $F(x)$ represents the size of the flow $x$, and $P(\chi_{H^j(\pi)})$ represents the $s-t$ flow polytope of $D$ when only the edges in $H^j(\pi)$ can be used. 

On $D_1$, the first algorithm is optimal and achieves a cumulative flow of $6 - \epsilon$, whereas the second algorithm is suboptimal and achieves a cumulative flow of $5$.
On $D_2$, the first algorithm is suboptimal and achieves a cumulative flow of $5 + 5\epsilon$, whereas the second algorithm is optimal and achieves a cumulative flow of $7$.
$D_1$ and $D_2$ can be combined to produce an instance in which neither algorithm is optimal.

In fact, there are instances of this problem in which the approximation factor of either greedy algorithm is $\Omega(|V|)$, as we showcase with the capacitated directed graph $D_3$ below.
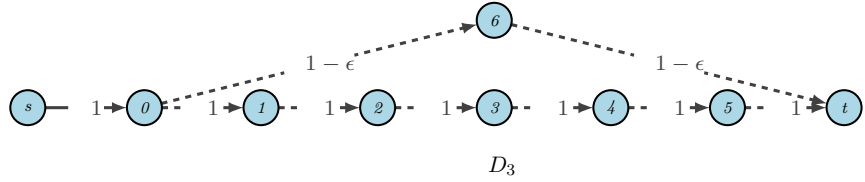
\begin{center}
\begin{minipage}{.7\linewidth}
\resizebox{\linewidth}{!}{
    \centering
    \begin{tikzpicture}
    \Text[x=0, y=-1]{$D_3$}
    \Vertex[x=-8, y=0,label=s]{s}
    \Vertex[x=-6, y=0,label=0]{0}
    \Vertex[x=-4, y=0, label=1]{1}
    \Vertex[x=-2, y=0, label=2]{2}
    \Vertex[x=0, y=0, label=3]{3}
    \Vertex[x=2, y=0, label=4]{4}
    \Vertex[x=4, y=0, label=5]{5}
    \Vertex[x=6, y=0,label=t]{t}
    \Vertex[x=0, y=1.5, label=6]{6}
    \Edge[Direct, Math, label=1, fontscale=4/3](s)(0)
    \Edge[Direct, Math, label=1, fontscale=4/3, style={dashed}](0)(1)
    \Edge[Direct, Math, label=1, fontscale=4/3, style={dashed}](1)(2)
    \Edge[Direct, Math, label=1, fontscale=4/3, style={dashed}](2)(3)
    \Edge[Direct, Math, label=1, fontscale=4/3, style={dashed}](3)(4)
    \Edge[Direct, Math, label=1, fontscale=4/3, style={dashed}](4)(5)
    \Edge[Direct, Math, label=1, fontscale=4/3, style={dashed}](5)(t)
    \Edge[Direct, Math, label=1-\epsilon, fontscale=4/3, style={dashed}](0)(6)
    \Edge[Direct, Math, label=1-\epsilon, fontscale=4/3, style={dashed}](6)(t)
    \end{tikzpicture}
}
\end{minipage}
\end{center}
The first algorithm can be unlucky and realize arcs allowing zero flow from $s$ to $t$ for the first $5$ steps and unit flow for only the last $3$ steps, whereas it is possible to have zero flow for only the first step, $1-\epsilon$ flow for the next $6$ steps, and unit flow for the last step.
The second algorithm behaves identically if in the first step it picks the maximum flow supported solely along the long path in $D_3$.
The long path in $D_3$ can be extended with any number of nodes to obtain the $\Omega(|V|)$ bound.
\end{example}

\subsection{Organization}
In Section~\ref{sec: general technique} we describe our general technique.
In Section~\ref{sec: algorithm for additive/linear functions} we show its exactness for the case of $f$ additive/$F$ linear, and we revisit our examples.
Lastly, in Section~\ref{sec: conclusion} we give concluding remarks.

\section{General Technique}
\label{sec: general technique}

For $a, b \in \mathbb{R}$ with $a \leq b$ denote $[a,b] := \{x \in \mathbb{R}: a \leq x \leq b\}$.
Let $P(\Pi) \subseteq [1, m]^m$ be the convex hull of permutations of $E$ (a.k.a., the permutahedron).
The following is well-known (see also Billera and Sarangarajan~\cite{billera1994combinatorics}, for example).
\begin{proposition}[Rado~\cite{rado1952inequality}]
$P(\Pi)$ has a linear inequality description given by
\begin{align*}
    P(\Pi) 
    = 
    \left\{
    y \in [1,m]^m : 
    \begin{array}{rlc}
              \sum_{i \in E} y_i &= \binom{m + 1}{2} &  \\
              \sum_{i \in S} y_i   &\leq \binom{m+1}{2} - \binom{m+1-|S|}{2}, & \forall S \subseteq E: S \neq \emptyset
    \end{array}
    \right\}.
\end{align*}
\end{proposition}
Since $\binom{m+1}{2} - \binom{m+1-|S|}{2} - \sum_{i \in S} y_i$ is a submodular set function, this description has a strongly polynomial time separation oracle~\cite{schrijver2000combinatorial} (a set function $f: 2^E \rightarrow \mathbb{R}$ is said to be submodular if $f(S) + f(T) \geq f(S \cup T) + f(S \cap T)$ for every $S, T \subseteq E$).
Alternatively, Goemans~\cite{goemans2015smallest} gives a an extended formulation with $\Theta(m \log m)$ extension complexity using the sorting network of Ajtai, Koml{\'o}s, and Szemer{\'e}di~\cite{ajtai1983sorting}.

We give a set of $O(m^2)$ linear inequalities that, given an extreme point $y \in P(\Pi)$, produce a sequence $h^1, h^2, \ldots, h^m \in [0,1]^m$ of characteristic vectors corresponding to the chain $\emptyset = H^0(y) \subset H^1(y) \subset H^2(y) \subset \cdots \subset H^m(y) = E$ of subsets of $E$.
Let $h_i^j$ denote the $i$th entry of the $j$th vector.
\begin{proposition}
\label{proposition: transformation}
Let $y \in P(\Pi)$ be an extreme point.
Then, $h^1, h^2, \ldots, h^m \in [0,1]^E$ is a sequence of characteristic vectors corresponding to the chain $\emptyset = H^0(y) \subset H^1(y) \subset H^2(y) \subset \cdots \subset H^m(y) = E$ if and only if the following inequalities are satisfied:
\begin{subequations}
\begin{align}
    && h_i^0                    &= 0,            & \forall i \in [m] \label{eq: transformation 0} \\
    && h_i^{j}                  &\leq h_i^{j+1},   & \forall j \in [m - 1], i \in [m] \label{eq: transformation 1} \\
    && \sum_{i \in [m]} h_i^{j} &=    j,           & \forall j \in [m] \label{eq: transformation 2} \\
    && \sum_{k \in [j]} h_i^{k} &\geq j - y_i + 1, & \forall j \in [m], \forall i \in [m] \label{eq: transformation 3}
\end{align}
\end{subequations}
\end{proposition}
\begin{proof}
Assume without loss of generality (up to rearranging) that $y_i = i$ for all $i \in [m]$.
 
First, suppose the inequalities are satisfied.
Consider the case in which $j = 1$.
(\ref{eq: transformation 3}) with $i = 1$ requires $h_1^1 \geq j - y_1 + 1 = 1 - 1 + 1 = 1$.
(\ref{eq: transformation 0}) and (\ref{eq: transformation 1}) further require $h_i^1 \geq 0$ for $i > 1$.
(\ref{eq: transformation 2}) requires $\sum_{i \in [m]} h_i^1 = 1$, and so $h_i^1 = 1$ for $i = 1$ and $h_i^1 = 0$ for $i > 1$ is the only candidate solution remaining.
Note that $h^1$ is the characteristic vector of $H^1(y)$.

By way of induction, suppose $h^j$ is the characteristic vector of $H^j(y)$ for some $1 \leq j < m$.
(\ref{eq: transformation 1}) and induction require $h_i^{j+1} \geq 1$ for $1 \leq i \leq j$ and $h_i^{j+1} \geq 0$ for $i > j$.
(\ref{eq: transformation 3}) with $i=j+1$ requires
\begin{align*}
    h_{j+1}^{j+1} = h_{j+1}^{1} + h_{j+1}^{2} + \cdots + h_{j+1}^{j} + h_{j+1}^{j+1} 
    \geq j + 1 - y_{j+1} + 1 
    \geq j + 1 - (j + 1) + 1
    = 1,
\end{align*}
where the first equality holds by induction.
(\ref{eq: transformation 2}) requires $\sum_{i \in [m]} h_{i}^{j+1} = j+1$, and so $h_i^{j+1} = 1$ for $1 \leq i \leq j+1$ and $h_i^{j+1} = 0$ for $i > j+1$ is the only candidate solution remaining.
Note that $h^{j+1}$ is the characteristic vector of $H^{j+1}(y)$.

Lastly, we need to show that letting $h^1, h^2, \ldots, h^m$ be characteristic vectors corresponding to the chain $\emptyset = H^0(y) \subset H^1(y) \subset H^2(y) \subset \cdots \subset H^m(y) = E$ indeed satisfies the inequalities.
(\ref{eq: transformation 0})-(\ref{eq: transformation 2}) are clearly satisfied.
It remains to show (\ref{eq: transformation 3}) is satisfied.
Take any $j \in [m]$ and $i \in [m]$.
If $i \geq j + 1$, then $y_i = i \geq j + 1$ and so $j - y_i + 1 \leq j - (j+1) + 1 = 0$ and the inequality is trivially satisfied.
If $i \leq j$, then $h_i^k = 1$ for all $i \leq k \leq j$ and so $\sum_{k \in [j]} h_i^k \geq j - i + 1= j - y_i + 1$.
\end{proof}
For notational convenience, we treat the linear system (\ref{eq: transformation 0})-(\ref{eq: transformation 3}) as a transformation $Z$ that maps $y$ to $h^1, h^2, \ldots, h^m$.
We first consider the case in which our subproblem is continuous of the form (\ref{eq: continuous optimization}).
We use $P(\Pi)$ and $Z$ to obtain
\begin{align}
\label{eq: relaxed permutatorial optimization - continuous}
    \max\{G(y) : y \in P(\Pi)\},
\end{align}
as a continuous relaxation of (\ref{eq: permutatorial optimization - continuous}), where $G: [1,m]^m \rightarrow \mathbb{R}$ is the continuous extension of $g$ given by
\begin{align*}
    G(y) := \sum_{j=1}^m \max\{F(x): Z(y) = h^1, h^2, \ldots, h^m \land x \in P(h^j) \}.
\end{align*}
In this way, the feasible region of (\ref{eq: relaxed permutatorial optimization - continuous}) and the feasible region of each subproblem in $G(y)$ admit (possibly exponential sized) linear inequality descriptions.

We now consider the case in which our subproblem is combinatorial of the form~(\ref{eq: combinatorial optimization}).
Let $P(\mF)$ be the convex hull of characteristic vectors corresponding to $\mF$.
Let $F: [0,1]^m \rightarrow \mathbb{R}$ be a continuous extension of $f$.
For example, if $f$ is additive with coefficients $w \in \mathbb{R}^m$ (meaning $f(S) = \sum_{i \in S} w_i$ for all $S \subseteq E$), we may simply let $F(x) = \sum_{i=1}^m w_i x_i$.
Then we obtain
\begin{align}
\label{eq: relaxed combinatorial optimization}
    \max\{F(x): x \in P(\mF)\}
\end{align}
as a continuous relaxation of (\ref{eq: combinatorial optimization}).
Note that $P(\mF)$ always admits a (possibly exponential sized) linear inequality description, and that when $f$ is additive this relaxation is exact.
Similarly, we use $F$, $P(\Pi)$, and $Z$ to obtain
\begin{align}
\label{eq: relaxed permutatorial optimization - combinatorial}
    \max\{G(y): y \in P(\Pi)\}
\end{align}
as a continuous relaxation of (\ref{eq: permutatorial optimization - combinatorial}), where $G: [1, m]^m \rightarrow \mathbb{R}$ is the continuous extension of $g$ given by
\begin{align*}
    G(y) := \sum_{j=1}^m \max\{F(x): Z(y) = h^1, h^2, \ldots, h^m \land x \leq h^j \land x \in P(\mathcal{F}) \}.
\end{align*}
As before, the feasible region of (\ref{eq: relaxed permutatorial optimization - combinatorial}) and the feasible region of each subproblem in $G(y)$ admit (possibly exponential sized) linear inequality descriptions.

\section{Algorithm for Additive/Linear Functions}
\label{sec: algorithm for additive/linear functions}

In this section we show that for $f$ additive/$F$ linear, the continuous relaxations of (\ref{eq: permutatorial optimization - combinatorial}) and (\ref{eq: permutatorial optimization - continuous}) given in Section~\ref{sec: general technique} are in fact exact.
Moreover, if the combinatorial/continuous subproblem can be solved in polynomial time via linear programming, then so can the corresponding permutatorial problem.

\begin{theorem}
\label{theorem: main - continuous}
If $F$ is linear and linear functions can be optimized over $P$ in polynomial time, then (\ref{eq: permutatorial optimization - continuous}) and (\ref{eq: relaxed permutatorial optimization - continuous}) are equivalent and can be solved in polynomial time via linear programming.
\end{theorem}
\begin{proof}
Since the terms of the summation in the objective function of (\ref{eq: relaxed permutatorial optimization - continuous}) are decoupled given $y \in P(\Pi)$, we may rewrite (\ref{eq: relaxed permutatorial optimization - continuous}) as
\begin{align}
\label{eq: lp - continuous 1}
    \max \left\{ \sum_{j=1}^m F(x^j) :
    y \in P(\Pi)
    \land Z(y) = h^1, h^2, \ldots, h^m  
    \land x^j \in P(h^j)
    \right\},
\end{align}
which can be solved in polynomial time since each set of inequalities describing the feasible region has either a compact description or a polynomial time separation oracle (in the case of $x^j \in P(h^j)$, we use the equivalence of separation and optimization and the assumption that linear functions can be optimized over $P$ in polynomial time).
Since the objective function and constraints of (\ref{eq: lp - continuous 1}) are linear, there exists an extreme point optimal solution.
Since in such a solution $y$ is an extreme point of $P(\Pi)$, Proposition~\ref{proposition: transformation} implies (\ref{eq: lp - continuous 1}) and 
\begin{align}
\label{eq: lp - continuous 2}
    \max \left\{ \sum_{j=1}^m F(x^j) :
    \pi \in \Pi
    \land x^j \in P(\chi_{H^j(\pi)})
    \right\}
\end{align}
are equivalent.
Lastly, note that (\ref{eq: lp - continuous 2}) and (\ref{eq: permutatorial optimization - continuous}) are equivalent since the terms of the summation in the objective function of (\ref{eq: permutatorial optimization - continuous}) are decoupled given $\pi \in \Pi$.
\end{proof}
We now prove our main result Theorem~\ref{theorem: main - combinatorial}: if $f$ is additive and linear functions can be optimized over $P(\mF)$ in polynomial time, then (\ref{eq: permutatorial optimization - combinatorial}) and (\ref{eq: relaxed permutatorial optimization - combinatorial}) are equivalent and can be solved in polynomial time via linear programming.
\begin{proof}[Proof~of~Theorem~\ref{theorem: main - combinatorial}]
Since $f$ is additive (say, with coefficients $w \in \mathbb{R}^m$), we may simply let $F(x) = \sum_{i=1}^m w_i x_i$, which is linear.
Then, (\ref{eq: relaxed permutatorial optimization - combinatorial}) is the special case of (\ref{eq: relaxed permutatorial optimization - continuous}) in which $F$ is linear and $P(h^j) = \{x \in \mathbb{R}^m : x \leq h^j \land x \in P(\mF)\}$.
Moreover, by the assumption that linear functions can be optimized over $P(\mF)$ in polynomial time, all conditions of Theorem~\ref{theorem: main - continuous} are met.
Therefore, (\ref{eq: relaxed permutatorial optimization - combinatorial}) is equivalent to the special case of (\ref{eq: permutatorial optimization - continuous}) in which $F$ is linear and $P(\chi_{H^j(\pi)}) = \{x \in \mathbb{R}^m : x \leq \chi_{H^j(\pi)} \land x \in P(\mF)\}$, and moreover can be solved in polynomial time via linear programming.
Lastly, note that since $f$ is additive, $\max\{F(x) : x \leq \chi_{H^j(\pi)} \land x \in P(\mF) \}$ is equivalent to $\max\{f(S) : S \in \mF(H^j(\pi))\}$, and so (\ref{eq: permutatorial optimization - continuous}) and (\ref{eq: permutatorial optimization - combinatorial}) are equivalent.
\end{proof}
\begin{remark}
More broadly, our proofs show that, for $f$ additive/$F$ linear, the time complexity of solving the combinatorial/continuous problem via the polyhedral approach (be it polynomial or not) extends to solving  the corresponding permutatorial problem via the polyhedral approach at the expense of $O(m^2)$ additional variables and linear inequalities.
For example, our technique can be incorporated within an integer linear programming framework.
\end{remark}

Our result extends, for example, to minimizing linear functions with non-negative coefficients over the dominant $\text{dom}(P(\mF))$ of $P(\mF)$, provided it has a compact description or a polynomial time separation oracle.
The dominant $\text{dom}(P) := \{x + y: x \in P, y \in \mathbb{R}_{\geq 0}^m\}$ of a polytope $P \subset \mathbb{R}^m$ is often used in optimization since it may have a simpler description than $P$ and since minimizing non-negative linear functions over $\text{dom}(P)$ is equivalent to minimizing non-negative linear functions over $P$.

We showcase our framework by revisiting the examples from Section~\ref{sec: introduction}.
\begin{example}[Matchings, cont.]
Let $G = (V, E)$ be a bipartite graph with weights $w: E \rightarrow \mathbb{Q}_{\geq 0}$.
Let $m = |E|$ so that we treat the edge set as $E = [m]$.
Then, by Theorem~\ref{theorem: main - combinatorial}, solving
\begin{subequations}
\begin{align}
    \maximize~&& \sum_{j=1}^m \sum_{i \in E} w_i x_i^j & & \\ 
              && y &\in P(\Pi) &  \\
              && h^1, h^2, \ldots, h^m  &= Z(y), &   \\
              && x^j  &\leq h^j, &  \forall j \in [m] \\
              && \sum_{i \in \delta(v)} x_i^j &\leq 1, & \forall j \in [m], \forall v \in V \label{eq: bipartite} \\
              && x^j &\geq 0, & \forall j \in [m]
\end{align}
\end{subequations}
is equivalent to solving (\ref{eq: permutatorial optimization - combinatorial}). 
If $G$ were not bipartite, we could replace (\ref{eq: bipartite}) with Edmonds'~\cite{edmonds1965maximum} description of the matching polytope.
\end{example}
\begin{example}[Flows, cont.]
\label{ex: flows, cont.}
Let $m = |A|$ so that we treat the arc set as $A = [m]$.
Then, by Theorem~\ref{theorem: main - continuous}, solving
\begin{subequations}
\begin{align}
    \maximize~&& \sum_{j=1}^m \sum_{i \in \delta^+(s)} f_u^j & & \\ 
              && y &\in P(\Pi) &  \\
              && h^1, h^2, \ldots, h^m  &= Z(y), &   \\
              && x^j  &= h^j, &  \forall j \in [m] \\
              && \sum_{i \in \delta^+(u)} f_i^j &= \sum_{i \in \delta^-(u)} f_i^j, & \forall j \in [m], \forall u \in V \setminus \{s,t\} \label{eq: mass conservation} \\
              && f_i^j &\leq c_i x_i^j, & \forall j \in [m], \forall i \in A \label{eq: capacity} \\
              && x^j, f^j &\geq 0, & \forall j \in [m]
\end{align}
\end{subequations}
is equivalent to solving (\ref{eq: permutatorial optimization - continuous}). 
For $u \in V$, $\delta^-(u)$ and $\delta^+(u)$ denote its incoming and outgoing arcs. 
\end{example}

\section{Conclusion}
\label{sec: conclusion}

We do not know of ``combinatorial'' algorithms for instances of (\ref{eq: permutatorial optimization - combinatorial}) or (\ref{eq: permutatorial optimization - continuous}) that can be solved in polynomial time via our polyhedral methods (e.g., matchings, flows).
We do not know how or if approximation algorithms for NP-hard instances of (\ref{eq: combinatorial optimization}) can be leveraged to obtain approximation algorithms for the corresponding instances of (\ref{eq: permutatorial optimization - combinatorial}).
As showcased in our examples, it not even clear how to leverage optimal solutions.
On a similar vein, we do not know how or if optimization methods for other classes of set functions $f$ (e.g., submodular) extend from (\ref{eq: combinatorial optimization}) to (\ref{eq: permutatorial optimization - combinatorial}).
If $\mF = 2^E$ (i.e., unconstrained optimization) and $f$ is monotone submodular, the greedy algorithm can be shown to be a $1/e$-approximation algorithm for (\ref{eq: permutatorial optimization - combinatorial}) using techniques from the greedy algorithm analysis of Nemhauser, Wolsey, and Fisher~\cite{nemhauser1978analysis}.
Contrast this guarantee with the much worse approximation ratio exhibited by the greedy algorithms on the third instance in Example~\ref{example: flows}.
We do not know whether $1/e$ is tight\textemdash contrast this with the tight $(1-1/e)$-approximation to $(\ref{eq: combinatorial optimization})$ achieved by the greedy algorithm when $f$ is monotone submodular and $\mF$ are the independent sets of a uniform matroid~\cite{nemhauser1978analysis}.
\begin{proposition}
If $\mF = 2^E$ and $f$ is monotone submodular, the greedy algorithm is a $1/e$-approximation algorithm to (\ref{eq: permutatorial optimization - combinatorial}).
\end{proposition}
\begin{proof}
Let $\pi^* := \argmax g(\pi)$ be an optimal ordering and let $\ALG$ be the ordering selected by the greedy algorithm.
For $j \in [m]$, let ${\pi^*}^j = \{{\pi^*}^{-1}(1), {\pi^*}^{-1}(2), \ldots, {\pi^*}^{-1}(j)\}$ and let $\ALG^j = \{\ALG^{-1}(1), \ALG^{-1}(2), \ldots, \ALG^{-1}(j)\}$.
The greedy algorithm analysis of~\cite{nemhauser1978analysis} together with $\mF = 2^E$ and $f$ monotone submodular imply $E = \argmax \{f(S): S \in \mF\}$ and so $f(E)(1-(1-1/m)^j) \leq f({\ALG}^j)$ for every $j \in [m]$. 
Summing over $j \in [m]$ we obtain $f(E)\sum_{j=1}^m(1-(1-1/m)^j) \leq \sum_{j=1}^m f({\ALG}^j) = g(\ALG)$.
Since $f$ is monotone, $f({\pi^*}^j) \leq f(E)$ for every $j \in [m]$ and so $g(\pi^*) = \sum_{j=1}^m f({\pi^*}^j)\leq m f(E)$.
This implies $g(\pi^*) \cdot \frac{1}{m} \sum_{j=1}^m(1-(1-1/m)^j) \leq g(\ALG)$.
Lastly, note that $\frac{1}{m} \sum_{j=1}^m(1-(1-1/m)^j) \rightarrow 1/e$ from above as $m \rightarrow \infty$.
\end{proof}

\section*{Acknowledgements}
Work partially supported by the National Science Foundation through Grant No.~1839346 and 1952011.
The authors would like to thank the anonymous referee for their careful review and suggestions.
The first author would like to thank Shriya Nagpal for thoughtful discussions.

\bibliography{bib}
\bibliographystyle{siam}

\end{document}